\documentclass{amsart}
\usepackage{all2018,pxfonts,color}
\newcommand{\Log}{\operatorname{Log}}
\newcommand{\cA}{\mathcal{A}}
\newcommand{\Zns}{Z_{\mathrm{ns}}}
\newcommand{\rT}{\mathrm{T}}
\newcommand{\rS}{\mathrm{S}}
\newcommand{\rR}{\mathrm{R}}

\newcommand{\bz}{\mathbf{z}}
\newcommand{\abs}{\operatorname{abs}}

\begin{document}

\title{The dimension of an amoeba}

\author{Jan Draisma}
\address{Universit\"at Bern, Mathematisches Institut, Sidlerstrasse 5,
3012 Bern, and Eindhoven University of Technology}
\email{jan.draisma@math.unibe.ch}

\author{Johannes Rau}
\address{Universit\"at T\"ubingen, Fachbereich Mathematik, 
Auf der Morgenstelle 10, 72076 T\"ubingen, Germany}
\email{johannes.rau@math.uni-tuebingen.de}

\author{Chi Ho Yuen}
\address{Universit\"at Bern, Mathematisches Institut, Alpeneggstrasse 22,
3012 Bern}
\email{chi.yuen@math.unibe.ch}

\maketitle

\begin{abstract}
Answering a question by Nisse and Sottile, we derive a formula for the
dimension of the amoeba of an irreducible algebraic variety.
\end{abstract}

\section{Introduction and main result}

Let $X \subseteq (\CC^*)^n$ be an irreducible, closed algebraic
subvariety. We define 
\[ \Log:(\CC^*)^n \to \RR^n, \quad 
(z_1,\ldots,z_n) \mapsto (\log|z_1|,\ldots,\log|z_n|) \] 
and 
$\cA(X):=\Log(X)$, the {\em amoeba} of $X$. The amoeba is the image of
the semi-algebraic set ({\em algebraic amoeba})
\[ |X|:=\{(|z_1|,\ldots,|z_n|) \mid (z_1,\ldots,z_n) \in X\} \subseteq
\RR_{>0}^n, \]
under a diffeomorphism and thus has an obvious notion of dimension,
denoted $\dim_\RR \cA(X)$.
Clearly, $\dim_\RR \cA(X) \leq 2 \dim_\CC X$. In \cite{Nisse18}, Nisse
and Sottile raise the question when this inequality is strict, as happens
in the following two examples.

\begin{ex}[hypersurfaces] \label{ex:Hypersurface}
Suppose that $n>2$ and that $X$ is a hypersurface. Then
$\dim_\RR \cA(X) \leq n<2(n-1)=2 \dim_\CC X$. \hfill $\clubsuit$
\end{ex}

\begin{ex}[torus-invariant varieties] \label{ex:TorusAction}
Suppose that $X$ is stable under a subtorus $\rS
\subseteq (\CC^*)^n$ of dimension $k>0$. Denote by $Y$ the image of $X$ in the algebraic
torus $(\CC^*)^n/\rS \cong (\CC^*)^{n-k}$. The map $X \to Y$
has fibers of complex dimension $k$, and the corresponding map $\cA(X)
\to \cA(Y)$ has fibers of real dimension $k$---namely, translates
of $\cA(\rS)$, which is a linear subspace of $\RR^n$ spanned by its
intersection wih $\QQ^n$. Thus we have
\[ \dim_\RR \cA(X)=k+\dim_\RR \cA(Y) \leq k+2 \dim_\CC Y =-k+2 \dim_\CC
X < 2\dim_\CC X. \quad \quad \quad \quad \clubsuit \]
\end{ex}

Our theorem says that these are two instances of the same phenomenon,
and that this phenomenon is responsible for all drops in dimension.

\begin{thm} \label{thm:Main}
Let $X \subseteq (\CC^*)^n$ be an irreducible, closed algebraic
subvariety. Then 
\begin{align*} \dim_\RR \cA(X) = \min \{\ &2 \dim_\CC X + 2 \dim_\CC \rT -
\dim_\CC \rS \mid \\
& \rT \subseteq \rS \subseteq (\CC^*)^n \text{ subtori and }
\rS\cdot(\overline{\rT \cdot X}) = \overline{\rT \cdot X}\  \}.
\end{align*}
An equivalent but more concise formula can then be given as
\begin{align*} \dim_\RR \cA(X) = \min \{2 \dim_\CC \overline{\rS \cdot X} - \dim_\CC \rS \mid
\rS \subseteq (\CC^*)^n \text{ subtorus}\}.\\
\end{align*}
\end{thm}

In this theorem, $\overline{\rT \cdot X}$ (resp. $\overline{\rS \cdot X}$) is the Zariski closure of
the set of all $tz$ with $t \in \rT$ (resp. all $rz$ with $r \in \rS$) and $z \in X$;
notice that whenever $\rS\cdot(\overline{\rT \cdot X}) = \overline{\rT \cdot X}$ as in the formula, 
the set is also equal to $\overline{\rS \cdot X}$.
Naturally, $\rS$ and $\rT$ may be taken zero-dimensional, in which case
we recover the upper bound $2 \dim_\CC X$.

\begin{ex}[hypersurfaces revisited] \label{ex:Hyper2}
If $X$ is a hypersurface, then most one-dimensional tori $\rT
\subseteq (\CC^*)^n$ will satisfy $\overline{\rT \cdot X}=(\CC^*)^n$
(see Lemma~\ref{lm:Torus}), so we may take $\rS=(\CC^*)^n$. The bound
in the theorem is $2(n-1)+2-n=n$. \hfill $\clubsuit$
\end{ex}

To motivate the structure of this paper, we now prove the easy inequality
$\leq$ in our main theorem.

\begin{proof}[Proof of $\leq$ in Theorem~\ref{thm:Main}.]
Let $\rT \subseteq \rS \subseteq (\CC^*)^n$ be subtori such that
$Y:=\overline{\rT \cdot X}$ is $\rS$-stable. Then
\[ \dim_\RR \cA(X) \leq \dim_\RR \cA(Y) 
\leq 2\dim_\CC Y - \dim_\CC \rS \leq 2(\dim_\CC X + \dim_\CC \rT) -
\dim_\CC \rS, \]
where the second equality follows from Example~\ref{ex:TorusAction}. 
\end{proof}

If we want equality to hold in the proof above, then we need that first,
$\dim_\CC Y=\dim_\CC X+ \dim_\CC \rT$; second, the bound in
Example~\ref{ex:TorusAction} for the pair $(Y,\rS)$ is tight; and
third, $\dim_\RR \cA(X) = \dim_\RR \cA(Y) = \dim_\RR (\cA(X) +
\cA(\rT))$. Our proof of Theorem~\ref{thm:Main} consists of first
finding a torus $\rT$ with the latter property
(see Section~\ref{sec:OneDim}):

\begin{prop} \label{prop:RealAction}
Let $X \subseteq (\CC^*)^n$ be a closed, irreducible variety. Then the
Zariski-closure $\overline{|X|}$ in $(\RR^*)^n$ of the algebraic amoeba 
is stable under a subtorus of the real algebraic torus $(\RR^*)^n$
of dimension at least $2 \dim_\CC X - \dim_\RR \cA(X)$.
\end{prop}

In particular, if the amoeba has dimension strictly less than $2 \dim_\CC
X$, then a positive-dimensional real torus acts on $\overline{|X|}$. Using
this positive-dimensional torus, we prove Theorem~\ref{thm:Main} by
induction in Section~\ref{sec:Proofs}. 

Theorem~\ref{thm:Main} implies \cite[Conjecture 4.4]{Nisse18}, which
proposes {\em near torus actions} (Definition~\ref{def:NearTorusAction}
below) as the only cause of dimension drops for the amoeba.

\begin{cor} \label{cor:ConjNearTorus}
For an irreducible, closed subvariety $X \subseteq (\CC^*)$,
we have 
\[ \dim_\RR \cA(X) < \min\{n, 2 \dim_\CC X\} \] 
if and only if 
some nontrivial subtorus $\rS \subseteq (\CC^*)^n$ has a near action on $X$.
\end{cor}

We conclude this introduction
with a relation to the tropical variety of $X$, also to be proved in
Section~\ref{sec:Proofs}.

\begin{cor} \label{cor:Tropical}
For any irreducible, closed subvariety $X \subseteq (\CC^*)$
the dimension $\dim_\RR \cA(X)$ is determined by the 
tropical variety $\Trop(X) \subseteq \RR^n$ of $X$ via 
\begin{align*} 
\dim_\RR &\cA(X)=
\min \{\ 2 \dim_\RR \Trop(X) + 2 \dim_\RR T - \dim_\RR S \mid
\\
& T \subseteq S \subseteq \RR^n \text{ rational linear subspaces with }
S+(T+\Trop(X))=T+\Trop(X)\ \},
\end{align*}
where a subspace of $\RR^n$ is called {\em rational} if it is spanned by vectors in $\QQ^n$. Similar to Theorem~\ref{thm:Main}, we have the equivalent formula 
\begin{align*} 
\dim_\RR \cA(X)=
&\min \{\ 2 \dim_\RR (S+\Trop(X)) - \dim_\RR S \mid S \subseteq \RR^n \text{ rational linear subspace}\}. 
\end{align*}
\end{cor}

\subsection*{Acknowledgments}
JD and CHY were partially and fully, respectively, supported by NWO
Vici grant entitled {\em Stabilisation in Algebra and Geometry}. JD
thanks Mounir Nisse and Frank Sottile for sharing their work on amoebas
at the Mittag-Leffler semester on {\em Tropical Geometry, Amoebas, and
Polyhedra}. All authors thank the Institut Mittag-Leffler for inspiring
working conditions. JR thanks J\"urgen Hausen for helpful discussions.

\section{In search of a positive-dimensional torus}
\label{sec:OneDim}

Throughout this section we fix an irreducible, closed subvariety $X
\subseteq (\CC^*)^n$. If $\dim_\RR \cA(X) < 2\dim_\CC X$, then we will
find a one-dimensional torus $\rT \subseteq (\CC^*)^n$ such that $\rT \cap
(\RR^*)^n$ preserves the Zariski-closure $\overline{|X|}$ and $\dim_\RR
(\cA(X)+\cA(\rT))=\dim_\RR \cA(X)$.

\subsection*{Preliminaries}
We write $S^1 \subseteq \CC^*$ for the unit circle. Recall that this
is a real form of the algebraic group $\CC^*$: indeed, tensoring the
coordinate ring $\RR[c,s]/(c^2+s^2-1)$ of $S^1$ with $\CC$ yields the
coordinate ring $\CC[c,s]/((c+is)(c-is)-1)$, which we recognise as the
coordinate ring of an algebraic torus with standard coordinate $c+is$;
moreover, the the inverse morphism $S^1 \to S^1,(c,s) \mapsto (c,-s)$
complexifies to the inverse morphism $\CC^* \to \CC^*, (c+is) \mapsto
1/(c+is)=(c-is)$; and similarly for the multiplication morphism $S^1
\times S^1 \to S^1$. Both $S^1$ and the other real form of $\CC^*$, the
real-algebraic group $\RR^*$, will play fundamental roles in our proof.

We write $(S^1)^n \subseteq (\CC^*)^n$, where the former is a real form
of the latter algebraic group. For $p \in (\CC^*)^n$ and $Q$ any subset
of $\CC^n$ we write $pQ$ for the for the result of coordinate-wise
multiplication of $p$ with each element of $Q$.  Writing $1$ for the
unit element in $(\CC^*)^n$ and $T_\bullet \bullet$ for
(real or complex) tangent spaces, we have 
\[ T_1 (\CC^*)^n = \CC^n = \RR^n  \oplus_\RR i \RR^n = 
T_1 \RR^n  \oplus_\RR T_1 (S^1)^n. \]
Component-wise multiplication by $p \in (\CC^*)^n$ yields
\[ T_p (\CC^*)^n = p \RR^n \oplus_\RR i p \RR^n 
= T_p p \RR^n \oplus_\RR T_p p (S^1)^n. \]
Note that $p^{-1} T_p p (S^1)^n$ is naturally identified with (the
same) $i \RR^n$ for all $p \in (\CC^*)^n$, and $p^{-1} T_p p \RR^n$
is identified with (the same) $\RR^n$ for all $p$.

Rather than directly working with the amoeba of $X$, we will work
with the algebraic amoeba $|X|$, the image of $X$ under the semi-algebraic map
\[ \abs:(\CC^*)^n \to \RR_{>0}^n,\ (z_1,\ldots,z_n) \mapsto
(|z_1|,\ldots,|z_n|).\]
The reason for this is that $|X|$ is, by real quantifier elimination,
a semi-algebraic set, hence analysable with methods from
real algebraic geometry. The following is immediate.

\begin{lm} \label{lm:dLog}
At $p \in (\CC^*)^n$, the derivative $d_p \Log$
(respectively, $d_p \abs$) sends the real vector
space $T_p p (S^1)^n$ to zero and an element $p v$ with $v \in \RR^n$
to $v$ (respectively, to $|p|v$).  \hfill $\square$
\end{lm}

\subsection*{Subvarieties of real tori}

We prove an auxiliary result on subvarieties of real tori. We will use
the term {\em real-Zariski} to refer to the real Zariski topology
on a real algebraic variety or, more generally, on a semi-algebraic set. We
write $\Zns$ for the nonsingular locus of a real algebraic variety.

\begin{lm} \label{lm:RationalSpace}
Let $Z$ be a real-Zariski-closed subset of $(S^1)^n \subseteq (\CC^*)^n$. Then
the real subspace $\sum_{p \in \Zns} p^{-1} T_p Z \subseteq i \RR^n$
is spanned by its intersection with $i \QQ^n$.
\end{lm}

\begin{proof}
That subspace is additive under union of irreducible components, so we
may assume that $Z$ is irreducible. Let $Z_\CC \subseteq (\CC^*)^n$
be the complexification of $Z$, an irreducible algebraic variety.
After multiplying with $p^{-1}$ for any fixed $p \in Z_\CC$ we may
assume that $1 \in Z_\CC$. By \cite[Proposition 2.2]{Borel91} there
exist a natural number $m$ and exponents $e_1,\ldots,e_m \in \{\pm 1\}$
such that the image $\rT$ of the multiplication map
\[ \mu_Z: Z_\CC^m \to (\CC^*)^n, \quad 
\bz=(z_1,\ldots,z_m) \mapsto z_1^{e_1} \cdots z_m^{e_m} \]
is a closed, connected algebraic subgroup $\rT$ of $(\CC^*)^n$, i.e., a
sub-torus. Since $\Zns$ is Zariski-dense in $Z_\CC$, there
exists a 
point $\bz=(z_1,\ldots,z_m) \in \Zns^m$ (no
complexification!) such that 
the complex-linear map $d_\bz \mu_Z: T_\bz Z_\CC^m
\to T_{\mu(\bz)} \rT$ is surjective. Now $\mu_Z$ is the
restriction to $Z_\CC^m$ of the multiplication map $\mu: ((\CC^*)^n)^m \to (\CC^*)^n$
with the same definition. We have 
\[ \mu=L_{\mu(\bz)} \circ \mu \circ (L_{z_1^{-1}} \times \cdots
\times L_{z_m^{-1}}), \] 
where $L_x$ is left multiplication with $x \in (\CC^*)^n$, and
accordingly, 
\[ d_\bz \mu = d_1 L_{\mu(\bz)} \circ d_{(1,\ldots,1)} \mu \circ 
(d_{z_1} L_{z_1^{-1}} \times \cdots \times d_{z_m}
L_{z_m^{-1}}). \]
Using that the derivative of multiplication is addition and
the derivative of inverse is negation, we find 
\[ d_{\mu(\bz)} L_{\mu(\bz)^{-1}} \circ d_\bz \mu_Z: T_\bz Z^m_\CC \to
T_1 \rT,\ (v_1,\ldots,v_m) \mapsto e_1z_1^{-1} v_1 + \cdots + e_m z_m^{-1}
v_m; \]
and by the choice of $\bz$ this map is surjective. For each $j$ we have $T_{z_j} Z_\CC=T_{z_j} Z \oplus_\RR i
T_{z_j} Z$ and the complex-linear map $d_{\mu(\bz)}
L_{\mu(\bz)^{-1}} \circ d_\bz \mu_Z$
sends the real direct sum $\bigoplus_j T_{z_j} Z$ surjectively onto $T_1 (\rT \cap
(S^1)^n)=(T_1 \rT) \cap (i\RR)^n=:Q$. Since $\rT$ is an
algebraic torus, $Q$ is spanned by its intersection with
$i\QQ^n$, and the space in the lemma contains $Q$.
Moreover, for all $z \in Z$ we have $z^{-1} T_z Z \subseteq
Q$, so that
the space in the lemma is in fact equal to $Q$.
\end{proof}

\subsection*{A real torus action}

We return to our irreducible variety $X \subseteq (\CC^*)^n$. By standard
results in real algebraic geometry, $X$ is also irreducible when regarded
as a real-algebraic variety of dimension $2 \dim_\CC X$. 
Then the semialgebraic set $|X|$ is irreducible in the sense that
its (real) Zariski closure in $\RR^n$ is irreducible. 
To see that, first note that the square
\[ |X|^2:=\{(|z_1|^2,\ldots,|z_n|^2) \mid (z_1,\ldots,z_n) \in X\} \subseteq
\RR_{>0}^n \]
is irreducible, since it is the image of $X$ under an algebraic morphism. 
Now, since the map $(x_1, \dots, x_n) \mapsto (x_1^2, \dots, x_n^2)$ on
$(\RR^*)^n$ is a finite flat morphism, there exists exactly one
irreducible component of the preimage of the Zariski closure $\overline{|X|^2}$
which intersects the positive orthant. Hence $|X|$ is irreducible. 

\begin{proof}[Proof of Proposition~\ref{prop:RealAction}]
For $q \in |X|$, write $Z_q:=q^{-1} X \cap (S^1)^n$, which is a real Zariski-closed
subset of $(S^1)^n$ such that $qZ_q=\abs^{-1}(q) \cap X$ is the fiber of 
$\abs|_X$ over $q$. By Sard's theorem, there is an open subset $U$ of $|X|$, 
dense in $|X|$ in the real Zariski-topology, such that $Z_q$ has dimension equal to the expected dimension 
$c:=2\dim_\CC X - \dim_\RR |X|=2\dim_\CC X - \dim_\RR \cA(X)$.
For each $q \in U$, define 
\[ Q_q:=\sum_{p \in (Z_q)_{\mathrm{ns}}} p^{-1} T_p Z_q
\subseteq i \RR^n, \]
which is a real vector space of dimension at least $c$, spanned by $Q_p
\cap i\QQ$ by Lemma~\ref{lm:RationalSpace}. 

Fix $q \in U$. For each $p \in Z_q$ we have $qp \in X$ and $qp(p^{-1}
T_p Z_q) \subseteq T_{qp} X$ and hence, since $X$ is a complex algebraic
variety, also $(qp) (i p^{-1} T_p Z_q) \subseteq T_{qp} X$. The space on
the left is contained in $qp \RR^n$, and hence, by Lemma~\ref{lm:dLog},
$d_{qp} \abs$ maps it onto $|qp| (i p^{-1} T_p Z_q) = q (i p^{-1}
T_p Z_q)$. We conclude that the latter space is contained in $T_q U$
for each $p \in Z_q$. Therefore,
\[ T_q U \supseteq q \sum_{p \in (Z_q)_{\mathrm{ns}}} i
p^{-1} T_p Z_q = q (i Q_q); \]
here $iQ_q \subseteq \RR^n$ is spanned by its intersection
with $\QQ^n$. 
Now for each vector space $R \subseteq \RR^n$ of dimension at least $c$
and spanned by its intersection with $\QQ^n$, the set
\[ V_R:=\{q \in U \mid T_q U \supseteq q R\} \]
is a real-Zariski-closed subset of $U$. There are only countably
many such $R$, and the above discussion shows that the closed sets
$V_R$ cover the semialgebraic set $U$. 

But then one of them must have
dimension equal to that of $U$, and in fact, since the Zariski closure of
$U$ is irreducible,
be equal to $U$. We conclude that there exists a real vector space
$R \subseteq \RR^n$, of dimension at least $c$ and spanned by $R \cap
\QQ^n$, such that $qR \subseteq T_q U$ for all $q \in U$. But then $q R
\subseteq T_q \overline{|X|}$ for all $q$ in the real algebraic variety
$\overline{|X|}$. Since $R$ is spanned by its intersection
with $\QQ^n$, there exists a real-algebraic torus $\rR_\RR \subseteq (\RR^*)^n$ 
with Lie algebra $R$. The sub-bundle of the tangent bundle of
$(\RR^*)^n$ that arises by differentiating the action of $\rR_\RR$ on $(\RR^*)^n$
is tangent to $\overline{|X|}$.  This implies
that $\overline{|X|}$ is $\rR_\RR$-stable.
\end{proof}

\section{Proofs of the main results} \label{sec:Proofs}

We begin with a lemma that was already used in the introduction
(Example~\ref{ex:Hyper2}).

\begin{lm} \label{lm:Torus}
Let $X \subseteq (\CC^*)^n$ be a closed, irreducible subvariety and
$\rS \subseteq (\CC^*)^n$ a subtorus. Then there exists a subtorus $\rT
\subseteq \rS$ with $\dim_\CC \rT=\dim_\CC \overline{\rS \cdot X}-\dim_\CC X$
such that $\overline{\rT \cdot X}=\overline{\rS \cdot X}$.
\end{lm}

\begin{proof}
  We prove the statement by induction on $k = \dim_\CC \overline{\rS \cdot X}-\dim_\CC X$.
	If $k=0$, then $\overline{\rS \cdot X} = X$ and $\rT = \{1\}$ will do.
	If $k>0$, choose a one-dimensional subtorus $\rR \subset \rS$ such that
	$\dim_\CC \overline{\rR \cdot X} > \dim_\CC X$. Such $\rR$ exists since otherwise
	$X$ would be invariant under all such $\rR$ and hence under $\rS$. 
	Then the statement follows from the induction assumption applied to 
	$X' = \overline{\rR \cdot X}$
	and a torus $\rS'$ such that $\rS = \rR \times \rS'$.
\end{proof}

We now use Proposition~\ref{prop:RealAction} to establish our dimension
formula for the (ordinary or algebraic) amoeba.

\begin{proof}[Proof of Theorem~\ref{thm:Main}]
Let $X \subseteq (\CC^*)^n$ be Zariski-closed and irreducible. Since we
have already proved the inequality $\leq$ of the theorem, it suffices to
establish the existence of subtori $\rT \subseteq \rS$ of $(\CC^*)^n$
such that $\overline{\rT \cdot X}$ is $\rS$-stable and $\dim_{\RR}
\cA(X)=2 \dim_{\CC} X + 2 \dim_{\CC} \rT - \dim_{\CC} \rS$. We proceed
by induction on $n$. For $n=0$ we have $X=(\CC^*)^0$ and we can take
$\rS=\rT=\{1\}$. So we assume that $n>0$ and that the statement holds
for subvarieties of tori of dimension $n-1$.

If $\dim_\RR \cA(X)=2 \dim_{\CC} X$, then we may take $\rT=\rS=\{1\}$ and
we are done. So we may assume that $\dim_\RR \cA(X)<2 \dim_{\CC} X$. Then,
by Proposition~\ref{prop:RealAction}, there exists a one-dimensional,
real-algebraic torus $\rR_\RR \subseteq (\RR^*)^n$ which stabilises
the Zariski-closure $\overline{|X|}$ of the algebraic amoeba. Let $\rR
\subseteq (\CC^*)^n$ be the complexification of $\rR_\RR$.  Then we
find an open subset $U \subseteq \cA(X)$ whose complement has positive
codimension such that $U$ is a smooth manifold with $\cA(\rR) \subseteq
T_u U$ for each $u \in U$ (use Lemma~\ref{lm:dLog} for the tangent
vectors coming from the action of $\rR_\RR$). We find that the fibres of
the map $U \to \RR^n/\cA(\rR)$ have dimension $1$, and this implies that
\[ 
\dim_\RR \overline{\cA(R)+\cA(X)}/\cA(R) = -1 + \dim_\RR
\cA(X).
\]
(We note that we are working with the closure with respect to the Euclidean
topology of $\RR^n$ in the above formula.)

Define 
\[ \tilde{X}:=\overline{\rR \cdot X}/\rR \subseteq (\CC^*)^n
/ \rR \cong (\CC^*)^{n-1}. \]
Then the previous equation implies that the amoeba
\[ \cA(\tilde{X})=\overline{\cA(R)+\cA(X)}/\cA(R) \]
has real dimension equal to $-1+\dim_\RR \cA(X)$.

By the induction hypothesis, there exist subtori $\tilde{\rT} \subseteq
\tilde{\rS}$ of $(\CC^*)^n/\rR$ such that $\overline{\tilde{\rT} \cdot
\tilde{X}}$ is $\tilde{\rS}$-stable and
\[ \dim_\RR \cA(\tilde{X}) = 2 \dim_{\CC} \tilde{X} + 2
\dim_\CC \tilde{\rT} - \dim_\CC \tilde{\rS}. \]
We distinguish two cases. First, assume that $X$ is not stable
under $\rR$, so that $\dim_\CC \tilde{X}=\dim_\CC X$. Then let $\rT,\rS$
be the pre-images in $(\CC^*)^n$ of
$\tilde{\rT},\tilde{\rS} \subseteq (\CC^*)^n/\rR$, respectively. Then
$\rT \subseteq \rS$ are subtori such that $\overline{\rT \cdot X}$
is $\rS$-stable, and we find
\begin{align*} 
\dim_\RR \cA(X) &=
1 + \dim_\RR \cA(\tilde{X}) \\
&= 1 + 2\dim_\CC \tilde{X} + 2\dim_\CC \tilde{\rT} - \dim_\CC \tilde{\rS} \\
&= 1 + 2\dim_\CC X + 2(-1 + \dim_\CC \rT) - (-1 + \dim_\CC \rS) \\
&= 2 \dim_\CC X + 2 \dim_\CC \rT - \dim_\CC \rS.
\end{align*}
Second, assume that $X$ is stable under $\rR$.  As before, let $\rS$
be the pre-image of $\tilde{\rS}$
in $(\CC^*)^n$, but now let $\rT$ be any torus in $(\CC^*)^n$
of complex dimension {\em equal} to $\dim_\CC \tilde{\rT}$ that
projects surjectively onto $\rT$. Using that $X$ is $\rR$-stable and
$\overline{\tilde{\rT} \cdot \tilde{X}}$ is $\tilde{\rS}$-stable,
we find that $\overline{\rT \cdot X}$ is $\rS$-stable. Furthermore, 
\begin{align*}
\dim_\RR \cA(X) &= 
1 + \dim_\RR \cA(\tilde{X}) \\
&= 1 + 2\dim_\CC \tilde{X} + 2\dim_\CC \tilde{\rT} - \dim_\CC \tilde{\rS} \\
&= 1 + 2(-1+\dim_\CC X) + 2 \dim_\CC \rT - (-1 + \dim_\CC \rS) \\
&= 2 \dim_\CC X + 2 \dim_\CC \rT - \dim_\CC \rS,
\end{align*}
as desired.

For the second formula, let $\rT,\rS$ be subtori as in the first
formula. We then have
\[ 
2 \dim_\CC \overline{\rS \cdot X} - \dim_\CC \rS
= 2 \dim_\CC \overline{\rT \cdot X} - \dim_\CC \rS 
\leq 2 \dim_\CC X + 2 \dim_\CC \rT - \dim_\CC \rS,\]
so the second formula is a lower bound to the first formula. Conversely,
if $\rS$ is any subtorus, then by Lemma~\ref{lm:Torus} there exists a
subtorus $\rT \subseteq \rS$ such that $\overline{\rT \cdot X}=\overline{\rS \cdot
X}$ and $\dim_\CC \rT=\dim_\CC \overline{\rS \cdot X}-\dim_\CC X$, and we find 
\[ 
2 \dim_\CC X + 2 \dim_\CC T - \dim_\CC \rS
= 2 \dim_\CC \overline{\rS \cdot X} - \dim_\CC \rS,
\]
hence the first formula is a lower bound to the second formula.
\end{proof}

\begin{ex}
We give an alternative proof of \cite[Theorem 4.5]{Nisse18}, which says that if
$\dim_\RR \cA(X)=\dim_\CC X$, then $X$ is a single orbit under a subtorus
of $(\CC^*)^n$.  Take a subtorus $\rS\subseteq (\CC^*)^n$ that achieves
the minimum in the second formula of Theorem~\ref{thm:Main}. Since we
always have $\dim_\CC \overline{\rS \cdot X}\geq \dim_\CC \rS, \dim_\CC
X$, from our choice of $\rS$ we have
\[ 
\dim_\CC X = 2 \dim_\CC \overline{\rS \cdot X} - \dim_\CC \rS \geq
\dim_\CC X.\] 
Hence $\dim_\CC \rS = \dim_\CC X = \dim_\CC \overline{\rS \cdot X}$. But $\overline{\rS \cdot X}$ is irreducible 
and contains both $X$ and an orbit of $\rS$, so $X$ must be equal to
such an orbit.  \hfill $\clubsuit$
\end{ex}

\subsection*{Near Torus Actions}

We start by reviewing Nisse-Sottile's notion of near torus actions
\cite[Definition 4.1]{Nisse18}.

\begin{de} \label{def:NearTorusAction}
  Let $X \subseteq (\CC^*)$ be an irreducible closed subvariety
	and $\rS \subseteq (\CC^*)^n$. We set $Y := (\rS \cdot X)/\rS$. 
	Then $\rS$ \emph{has a near action on $X$} if 
	\begin{align*} 
		2 \dim_\CC X &> \dim_\CC \rS + 2 \dim_\CC Y &\text{ and } & & n &> \dim_\CC \rS + 2 \dim_\CC Y.
	\end{align*}
\end{de}

We now show that, as conjectured in \cite{Nisse18},
unexpected amoeba dimension is equivalent to a near torus
action. The implication $\Leftarrow$ is \cite[Theorem 4.3]{Nisse18}.

\begin{proof}[Proof of Corollary~\ref{cor:ConjNearTorus}]
For any subtorus $\rS \subseteq (\CC^*)^n$, setting $Y := \overline{\rS \cdot X}/ \rS$
we have $\dim_\CC(Y) = \dim_\CC(\rS \cdot X) - \dim_\CC(\rS)$.
Note that $2 \dim_\CC(\rS \cdot X) - \dim_\CC(\rS) < 2 \dim_\CC(X)$ and
$2 \dim_\CC(\rS \cdot X) - \dim_\CC(\rS) < n$ imply $\rS \neq \{1\}$ and
$\rS \neq (\CC^*)^n$, respectively.
Hence the statement follows directly 
from the second formula of Theorem~\ref{thm:Main}.
In particular, in case of a dimension drop, a torus $\rS$
providing the minimum in this formula has a near action on $X$.
\end{proof}

\subsection*{Proof of Corollary~\ref{cor:Tropical}}

We start by presenting a well-known fact in tropical geometry.

\begin{lm} \label{lem:LinealityTorusAction}
  Let $X \subseteq (\CC^*)$ be an irreducible (in particular, reduced) closed subvariety
	and denote by $\Trop(X) \subseteq \RR^n$ its tropicalisation. Let $\rS \subset (\CC^*)^n$ 
	be a subtorus and $S = \Trop(\rS) = \cA(\rS) \subset \RR^n$ the associated (rational) linear subspace.
	Then $\rS \cdot X =X$ if and only if $S + \Trop(X) = \Trop(X)$.  
\end{lm}

\begin{proof}
By basic tropical geometry, $\Trop(\overline{\rS \cdot X}) = S + \Trop(X)$.
	Hence $\rS \cdot X =X$ implies $S + \Trop(X) = \Trop(X)$. 
	Let us assume $S + \Trop(X) = \Trop(X)$ now. Note that
	for irreducible varieties $Y \subseteq (\CC^*)^n$, we have
	$\dim_\CC Y = \dim_\RR \Trop(Y)$, see \cite[Theorem A]{Bieri84}.
	Since both $X$ and $\overline{\rS \cdot X}$ are irreducible, 
	it follows that $\dim_\CC \overline{\rS \cdot X} = \dim_\CC X$.
	Since $X \subseteq \overline{\rS \cdot X}$, this implies
	$X = \overline{\rS \cdot X} = \rS \cdot X$. 
\end{proof}

\begin{proof}[Proof of Corollary~\ref{cor:Tropical}]
By Lemma~\ref{lem:LinealityTorusAction}, the pairs
of subtori $\rT \subseteq \rS \subseteq (\CC^*)^n$
such that $\rS \cdot (\overline{\rT \cdot X}) = \overline{\rT \cdot X}$ are in bijection
to the pairs of rational linear subspaces $T \subseteq S \subseteq \RR^n$ 
such that $S+(T+\Trop(X))=T+\Trop(X)$, via $T= \Trop(\rT)$, $S= \Trop(\rS)$.
Using the relation $\dim_\CC Y = \dim_\RR \Trop(Y)$ again,
we have
\begin{align*} 
	2\dim_\CC X + 2 \dim_\CC \rT -
	\dim_\CC \rS 
	= 2 \dim_\RR \Trop(X) + 2 \dim_\RR T - \dim_\RR S.
\end{align*}
Hence the two minima agree.
The second formula follows similarly as in the proof of
Theorem~\ref{thm:Main}.
\end{proof}

We conclude this paper with a question on computability.

\begin{que}
Does there exist an algorithm that, on input a balanced, pure-dimensional,
rational polyhedral complex $\Sigma \subseteq \RR^n$ which is connected
in codimension $1$, computes the expression
\[ \min \{2\dim_\RR (S+\Sigma)-\dim_\RR S \mid S \subseteq \RR^n \text{ rational
subspace} \} \]
from Corollary~\ref{cor:Tropical}?
\end{que}

The first term is the maximum, over all maximal cones $C$ of $\Sigma$,
of $\dim_\RR \Sigma + \dim_\RR S - \dim_\RR (\langle C \rangle_\RR \cap S)$, and
hence it is minimised by an $S$ have certain incidences with given
linear subspaces of $\RR^n$. If the rationality assumption is dropped,
then real quantifier elimination answers the question in the affirmative.
However, similar incidence problems often have real but no rational
solutions. For instance, a classical result in enumerative geometry
says that the number of two-dimensional subspaces in $\RR^4$ (lines in
projective three-space) that nontrivially intersect $4$ given two-dimensional subspaces in general
position is either zero (in which case there are two complex conjugate
solutions) or two. In the latter case, even if the four given spaces
are rational, the two solutions will typically not be. 
We do not know
whether the existence of rational solutions for such incidence problems
is decidable in general, nor whether the additional conditions on $\Sigma$
force that real solutions imply rational solutions.
On the other hand, if $X$ is a variety given by equations with
coefficients in, say, some number field, then of course, by real
quantifier elimination, there does exist an algorithm for computing
$\dim_\RR \cA(X)=\dim_\RR |X|$.

\bibliographystyle{alpha}
\bibliography{diffeq}

\end{document}